\documentclass[11pt]{amsart}
\usepackage{amsmath,amsfonts,amssymb,amsthm}
\usepackage{mathrsfs,geometry}
\usepackage{url}

\newcommand{\iy}{\ensuremath{\infty}}
\newcommand{\R}{\ensuremath{\mathbb{R}}}
\newcommand{\C}{\ensuremath{\mathbb{C}}}

\newcommand{\N}{\ensuremath{\mathbb{N}}}

\newcommand{\W}{\ensuremath{\mathcal{W}}}
\newcommand{\A}{\ensuremath{\mathcal{A}}}
\renewcommand{\S}{\ensuremath{\mathcal{S}}}
\renewcommand{\H}{\ensuremath{\mathcal{H}}}
\newcommand{\K}{\ensuremath{\mathcal{K}}}
\newcommand{\F}{\ensuremath{\mathcal{F}}}
\newcommand{\B}{\ensuremath{\mathcal{B}}}

\newcommand{\ind}{\ensuremath{\mathbf{1}}}

\newcommand{\e}{\varepsilon}
\renewcommand{\phi}{\varphi}
\renewcommand{\leq}{\leqslant}
\renewcommand{\geq}{\geqslant}
\newcommand{\Id}{\mathrm{Id}}

\providecommand{\invstackrel}[2]{\mathop{#1}\limits_{#2}}

\providecommand{\norm}[1]{\lVert#1\rVert}
\providecommand{\scalar}[2]{\langle #1, #2 \rangle}
\newcommand{\transp}{\top}
\DeclareMathOperator{\Fock}{\Phi}
\DeclareMathOperator{\TFock}{T \Phi}

\newtheorem{theo}{Theorem}

\newtheorem{pr}{Proposition}
\newtheorem{lemma}{Lemma}

\newtheorem{cor}{Corollary}

\newtheorem*{theo*}{Theorem}

\author[S.~Attal]{St\'ephane ATTAL}
\address{
Universit\'e de
Lyon, Universit\'e de Lyon 1, Institut Camille Jordan, CNRS UMR
5208, 43, Boulevard du 11 novembre 1918, 69622 Villeurbanne Cedex,
France} 
\email{attal@math.univ-lyon1.fr, nechita@math.univ-lyon1.fr}

\author[I.~Nechita]{Ion NECHITA}

\title{Discrete approximation of the free Fock space}
 %Central limit and other weak theorems

\begin{document}

\begin{abstract}
We prove that the free Fock space ${\F}(\R^+;\C)$, which is very commonly used in Free Probability Theory, is the continuous free product of copies of the space $\C^2$. We describe an explicit embedding and approximation of this continuous free product structure by means of a discrete-time approximation: the free toy Fock space, a countable free product of copies of $\C^2$. We show that the basic creation, annihilation and gauge operators of the free Fock space are also limits of elementary operators on the free toy Fock space. When applying these constructions and results to the probabilistic interpretations of these spaces, we recover some discrete approximations of the semi-circular Brownian motion and of the free Poisson process. All these results are also extended to the higher multiplicity case, that is, ${\F}(\R^+;\C^N)$ is the continuous free product of copies of the space $\C^{N+1}$.
\end{abstract}
\maketitle
\section{Introduction}
In \cite{attal-art} it is shown that the symmetric Fock space $\Gamma_s(L^2(\R^+;\C))$ is actually  the continuous tensor product $\otimes_{t\in\R^+}\C^2$. This result is obtained by means of an explicit embedding and approximation of the space $\Gamma_s(L^2(\R^+;\C))$ by countable tensor products $\otimes_{n\in h\N}\C^2$, when $h$ tends to 0. The result contains explicit approximation of the basic creation, annihilation and second quantization operators by means of elementary tensor products of 2 by 2 matrices. 

When applied to probabilistic interpretations of the corresponding spaces (e.g. Brownian motion, Poisson processes, ...), one recovers well-known approximations of these processes by random walks. This means that these different probabilistic situations and approximations are all encoded by the approximation of the three basic quantum noises: creation, annihilation and second quantization operators.

These results have found many interesting applications and developments in quantum statistical mechanics, for they furnished a way to obtain quantum Langevin equations describing the dissipation of open quantum systems as a continuous-time limit of basic Hamiltonian interactions of the system with the environment: repeated quantum interactions (cf \cite{attal-pautrat2, BJM, BP} for example). 

When considering the fermionic Fock space, even if it has not been written anywhere, it is easy to show that a similar structure holds, after a Jordan-Wigner transform on the spin-chain.

\smallskip 
It is thus natural to wonder if, in the case of the free Fock space, a similar structure, a similar approximation and similar probabilistic interpretations exist. Whereas the continuous tensor product structure of the bosonic Fock space exhibit its natural ``tensor-independence'' structure, it is natural to think that the free Fock space will exhibit a similar property with respect to the so-called ``free-independence'', as defined in Free Probability Theory. 

The key of our construction relies on the so-called ``free products of Hilbert spaces''. We needed to make explicit the constructions of countable free products, as a first step. Then, by an approximation method, to define the structure of continuous free products of Hilbert spaces. This structure appears to be exactly the natural one which describes the free Fock space and its basic operators.

\section{Free probability and the free Fock space}

Let us start by recalling the general framework of non commutative probability theory. A non commutative probability space is a couple $(\A, \phi)$, where $\A$ is a complex $*-$algebra (in general non commutative) and $\phi$ is a faithful positive linear form such that $\phi(1)=1$. We shall call the elements of $\A$ non commutative random variables. The distribution of a family $(x_i)_{i \in I}$ of self-adjoint random variables of $\A$ is the application which maps any non-commutative polynomial $P \in \C\langle X_i | i \in I \rangle$ to its moment $\phi(P((x_i)_{i \in I}))$. Thus, the map $\phi$ should be considered as the analogue of the expectation from classical probability theory. From this abstract framework, one can easily recover the setting of classical probability theory by considering a commutative algebra $\A$ (see \cite{hiai-petz, nica-speicher, voic-LN}).

In order to have an interesting theory, one needs a notion of independence for non commutative probability spaces. However, classical (or tensor) independence is not adapted in this more general setting. \emph{Free independence} was introduced by Voiculescu in the 1980's in order to tackle some problems in operator theory, and has found many applications since, mainly in random matrix theory. Freeness provides rules for computing mixed moments of random variables when only the marginal distributions are known. More precisely, unital sub-algebras $(\A_i)_{i \in I}$ of $\A$ are called \emph{free} (or \emph{freely independent}) if $\phi(a_1\cdots a_n)=0$ for all $n \in \N$ and $a_i \in \A_{j(i)}$ whenever $\phi(a_i)=0$ for all $i$ and neighboring $a_i$ do not come from the same sub-algebra: $j(1) \neq j(2) \neq \cdots \neq j(n)$. This definition allows one to compute mixed moments of elements coming from different algebras $\A_i$, when only the distributions inside each algebra $\A_i$ are known. Note that freeness is a highly non commutative property: two free random variables commute if and only if they are constant.

A remarkable setting in which freeness appears naturally is provided by creation and annihilation operators on the full Fock space. Let us now briefly describe this construction. Consider a complex Hilbert space $\H$ and define
\[\F(\H) = \bigoplus_{n=0}^\iy \H^{\otimes n},\] 
where $\H^{\otimes 0}$ is a one-dimensional Hilbert space we shall denote by $\C\Omega$. $\Omega \in \F(\H)$ is a distinguished norm one vector which is called the \emph{vacuum vector} and it will play an important role in what follows. 
For each $f \in \H$, we define the left \emph{creation} operator $\ell(f)$ and the left \emph{annihilation} operator $\ell^*(f)$ by
\begin{align*}
l(f)\Omega = f, & \quad l(f)e_1 \otimes \cdots \otimes e_n = f \otimes e_1 \otimes \cdots \otimes e_n;\\
l^*(f)\Omega = 0, & \quad l^*(f)e_1 \otimes \cdots \otimes e_n = \scalar{f}{e_1} e_2 \otimes \cdots \otimes e_n.\\
\end{align*}
For every $T \in \B(\H)$, the \emph{gauge (or second quantization) operator} $\Lambda(T) \in \B(\F(\H))$ is defined by
\[\Lambda(T)\Omega = 0,  \quad \Lambda(T)e_1 \otimes \cdots \otimes e_n = T(e_1) \otimes e_2 \otimes \cdots \otimes e_n.\]
All these operators are bounded, with $\norm{l(f)} = \norm{l^*(f)} = \norm{f}$ and $\norm{\Lambda(T)} = \norm{T}$.	
On the space $\B(\F(\H))$ of bounded operators on the full Fock space we consider the vector state given by the vacuum vector
\[\tau(X) = \scalar{\Omega}{X\Omega}, \quad X \in \B(\F(\H)).\]
The usefulness of the preceding construction when dealing with freeness comes from the following result (\cite{nica-speicher}).
\begin{pr}\label{pr:free_orth}
Let $\H$ be a complex Hilbert space and consider the non commutative probability space $(\B(\F(\H)), \tau)$. Let $\H_1, \ldots, \H_n$ be a family of orthogonal subspaces of $\H$, and, for each $i$, let $\A_i$ be the unital $*$-algebra generated by the set of operators
\[\{l(f) | f \in \H_i\} \cup \{\Lambda(T) | T \in \B(\H), T(\H_i) \subset \H_i \text{ and } T \text{ vanishes on }\H_i^\bot\}.\]
Then the algebras $\A_1, \ldots, \A_n$ are free in $(\B(\F(\H)), \tau)$.
\end{pr}

In the present note, we shall be concerned mostly with the case of $\H = L^2(\R_+; \C)$, the complex Hilbert space of square integrable complex valued functions; in Section \ref{sec:multi} we shall consider the more general case of  $L^2(\R_+; \C^N)$. Until then, we put $\Fock = \F(L^2(\R_+; \C))$, and we call this space the \emph{free (or full) Fock space}. An element $f \in \Fock$ admits a decomposition $f=f_0 \Omega + \sum_{n \geq 1} f_n$, where $f_0 \in \C$ and $f_n \in L^2(\R_+^n)$. In this particular case we shall denote the creation (resp. annihilation) operators by $a^+$ (resp. $a^-$):
\begin{align*}
a^+(h)\Omega = h, &\quad a^+(h)f_n = [(x_1, x_2, \ldots, x_n ,x_{n+1}) \mapsto h(x_1)f_n(x_2, \ldots, x_{n+1})],\\
a^-(h)\Omega = 0, &\quad a^-(h)f_n = [(x_2, \ldots,x_n) \mapsto \int{h(x)f_n(x, x_2\ldots, x_n) dx}],
\end{align*}
where $h$ is an arbitrary function of $L^2(\R_+)$. 
For a bounded function $b \in L^\iy(\R_+)$, let $a^\circ(b)$ be the gauge operator associated to the operator of multiplication by $b$:
\[a^\circ(b)\Omega = 0, \quad a^\circ(b)f_n = [(x_1, x_2, \ldots, x_n) \mapsto b(x_1)f_n(x_1, \ldots, x_n)],\]
and $a^\times(b)$ the scalar multiplication by $\int b$:
\[a^\times(b)\Omega =  \int b(x)dx\cdot \Omega, \quad a^\times(b)f_n = [(x_1, x_2, \ldots, x_n) \mapsto \left(\int b(x)dx\right) \cdot  f_n(x_1, \ldots, x_n)].\]
Finally, we note $\ind_t = \ind_{[0,t)}$ the indicator function of the interval $[0,t)$ and, for all $t \in \R_+$ and $\e \in \{+, -, \circ, \times\}$, we put $a_t^\e = a^\e(\ind_{[0, t)})$. Obviously, $a_t^\times = t\cdot \Id$.

\section{The free product of Hilbert spaces}

In the previous section we have seen that the algebras generated by creation, annihilation and gauge operators acting on orthogonal subspaces of a Hilbert space $\H$ are free in the algebra of bounded operators acting on the full Fock space $\F(\H)$. However, one would like, given a family of non commutative probability spaces, to construct a larger algebra which contains the initial algebras as sub-algebras which are freely independent. In classical probability (usual) independence is achieved by taking the tensor products of the original probability spaces. This is the reason why classical independence is sometimes called tensor independence. In the free probability theory, there is a corresponding construction called the \emph{free product}. Let us recall briefly this construction (see \cite{vdn, voic-LN} for further details).

Consider a family $(\H_i, \Omega_i)_{i \in I}$ where the $\H_i$ are complex Hilbert spaces and $\Omega_i$ is a distinguished norm one vector of $\H_i$. Let $\K_i$ be the orthocomplement of $\Omega_i$ in $\H_i$ and define the free product
\begin{equation}\label{eq:free_product_hilbert}
(\H, \Omega) = \invstackrel{\bigstar}{i \in I} (\H_i, \Omega_i) := \C \Omega \oplus \invstackrel{\bigoplus}{n \geq 1}\invstackrel{\bigoplus}{i_1 \neq i_2 \neq \cdots \neq i_n} \K_{i_1} \otimes \cdots \otimes \K_{i_n},
\end{equation}
where the direct sums are orthogonal and, as usual, $\norm \Omega = 1$. 
%By identitfying each $\Omega_i$ with $\Omega$, 
As in \cite{voic-LN}, we proceed with the identification of the algebras of bounded operators $\B(\H_i)$ inside $\B(\H)$. To this end, we shall identify an operator $T_i \in \B(\H_i)$, with the operator $\tilde T_i \in \B(\H)$ which acts in the following way: 
\begin{align}
\label{eq:op_embed_first}\tilde T_i(\Omega) &= T_i(\Omega_i)\\
\tilde T_i(k_i \otimes k_{j_1} \otimes \cdots \otimes k_{j_n}) &=  T_i(k_i) \otimes k_{j_1} \otimes \cdots \otimes k_{j_n} \\
\label{eq:op_embed_last}\tilde T_i(k_{j_1} \otimes \cdots \otimes k_{j_n}) &=  T_i(\Omega_i) \otimes k_{j_1} \otimes \cdots \otimes k_{j_n} 
\end{align}
where $j_1 \neq i$ and we identify an element of $\H_i$ with the corresponding element of $\H$. The main interest of this construction is the following straightforward result.

\begin{pr}
The algebras $\{\B(\H_i)\}_{i \in I}$ are free in $(\B(\H), \phi)$.
\end{pr}
\begin{proof}
Consider a sequence $T_{i(1)}, \ldots, T_{i(n)}$ of elements of $\B(\H_{i(1)}), \ldots, \B(\H_{i(n)})$ respectively such that $i(1) \neq i(2) \neq \cdots \neq i(n)$ and $\scalar{\Omega_{i(k)}}{T_{i(k)}\Omega_{i(k)}} = 0$ for all $k$. By the definition of freeness, it suffices to show that $\scalar{\Omega}{\tilde T_{i(1)} \cdots \tilde T_{i(n)}\Omega} = 0$. Using the previously described embedding, we get $ \tilde T_{i(n)} \Omega = T_{i(n)} \Omega_{i(n)}$. Since $i(n-1) \neq i(n)$ and $ \tilde T_{i(n)} \Omega \notin \C\Omega$, it follows that $\tilde T_{i(n-1)} \tilde T_{i(n)} \Omega =  [T_{i(n-1)} \Omega_{i(n-1)}] \otimes [T_{i(n)} \Omega_{i(n)}]$. Continuing this way, it is easy to see that $\tilde T_{i(1)} \cdots \tilde T_{i(n)} \Omega =  [T_{i(1)} \Omega_{i(1)}] \otimes \cdots \otimes [T_{i(n)} \Omega_{i(n)}]$, and the conclusion follows. 
\end{proof}

We look now at the free Fock space of a direct sum of Hilbert spaces. In the symmetric case (see \cite{attal-art}), it is known that one has to take the tensor product of the symmetric Fock spaces in order to obtain the Fock space of the sum. The free setting admits an analogue \emph{exponential property}, where instead of the tensor product one has to use the free product introduced earlier.
\begin{lemma} Consider a family of orthogonal Hilbert spaces $(\H_i)_{i \in I}$. Then
\begin{equation}\label{eq:exp}
\F(\oplus_{i \in I} \H_i) = \bigstar_{i \in I} \F(\H_i).
\end{equation}
\end{lemma}
\begin{proof}
Fix for each $\H_i$ an orthonormal basis $(X^j(i))_{j \in B(i)}$. Then, an orthonormal basis of $\F(\oplus \H_i)$ is given by $\{\Omega\} \cup \{X^{j_1}(i_1) \otimes \cdots \otimes X^{j_n}(i_n)\}$, where  $n \geq 1$, $i_k \in I$  and $j_k \in B(i_k)$ for all $1 \leq k \leq n$. One obtains a Hilbert space basis of $\bigstar \F(\H_i)$ by grouping adjacent elements of $X^{j_1}(i_1) \otimes \cdots \otimes X^{j_n}(i_n)$ with the same $i$-index (i.e. belonging to the same $\H_i$). Details are left to the reader.
\end{proof}

\section{The free toy Fock space}\label{sec:free_toy_fock}

In this section we introduce the \emph{free toy Fock space}, the main object of interest in our paper. 
%Formally, the free toy Fock space is a countable free product of copies of $\C^2$. 
From a probabilistic point of view, it is the ``smallest'' non commutative probability space supporting a free identically distributed countable family of Bernoulli random variables (see Section \ref{sec:free_prob}). 

The free toy Fock space is a countable free product of two-dimensional complex Hilbert spaces: in equation (\ref{eq:free_product_hilbert}), take $\H_i=\C^2$ for all $i$. In order to keep track of which copy of $\C^2$ we are referring to, we shall label the $i$-th copy with $\C^2_{(i)}$. Each copy is endowed with the canonical basis $\{\Omega_i = (1, 0)^\transp, X_i = (0, 1)^\transp\}$. Since the orthogonal space of $\C\Omega_i$ is simply $\C X_i$, we obtain the following simple definition of the free toy Fock space $\TFock$:
\[(\TFock, \Omega) := \invstackrel{\bigstar}{i \in \N} (\C^2_{(i)}, \Omega_i) = \C \Omega \oplus \invstackrel{\bigoplus}{n \geq 1}\invstackrel{\bigoplus}{i_1 \neq \cdots \neq i_n} \C X_{i_1} \otimes \cdots \otimes \C X_{i_n},\]
where, as usual, $\Omega$ is the identification of the vacuum reference vectors $\Omega_i$ ($\norm{\Omega} = 1$). Note that the orthonormal basis of $\TFock$ given by this construction is indexed by the set of all finite (eventually empty) words with letters from $\N$ with the property that neighboring letters are distinct. More formally, a word $\sigma = [i_1, i_2, \ldots, i_n] \in \N^n$ is called \emph{adapted} if $i_1 \neq i_2 \neq \cdots \neq i_n$. By convention, the empty word $\emptyset$ is adapted. We shall denote by $\W_n$ (resp. $\W_n^*$) the set of all words (resp. adapted words) of size $n$ and by $\W$ (resp. $\W^*$) the set of all words (resp. adapted words) of finite size (including the empty word).  For a word $\sigma = [i_1, i_2, \ldots, i_n]$, let $X_\sigma$ be the tensor $X_{i_1} \otimes X_{i_2} \otimes \cdots \otimes X_{i_n}$ and put $X_\emptyset = \Omega$. With this notation, an orthonormal basis of $\TFock$ is given by $\{X_\sigma\}_{\sigma \in \W^*}$.

We now turn to operators on $\C^2_{(i)}$ and their embedding into $\B(\TFock)$. We are interested in the following four operators acting on $\C^2$:
\[a^+=\begin{bmatrix}
0&0\\
1&0
\end{bmatrix}, \quad a^-=\begin{bmatrix}
0&1\\
0&0
\end{bmatrix}, \quad a^\circ=\begin{bmatrix}
0&0\\
0&1
\end{bmatrix}, \quad a^\times=\begin{bmatrix}
1&0\\
0&0
\end{bmatrix}.
\]

For $\e \in \{+, -, \circ, \times\}$, we shall denote by $a_i^\e$ the image of $a^\e$ acting on the $i$-th copy of $\C^2$, viewed (by the identification described earlier in eq. (\ref{eq:op_embed_first}) - (\ref{eq:op_embed_last})) as an operator on $\TFock$. The action of these operators on the orthonormal basis of $\TFock$ is rather straightforward to compute ($\sigma = [\sigma_1, \ldots, \sigma_n]$ is an arbitrary non-empty adapted word and $\ind$ is the indicator function):
\begin{align}\label{eq:op_toy_Fock_first}
a_i^+ \Omega = X_i, \quad &a_i^+ X_\sigma = \ind_{\sigma_1 \neq i}X_{[i, \sigma]};\\
a_i^- \Omega = 0, \quad &a_i^- X_\sigma = \ind_{\sigma_1 = i}X_{[\sigma_2, \ldots, \sigma_n]};\\
\label{eq:op_toy_Fock_next_to_last}a_i^\circ \Omega = 0, \quad &a_i^\circ X_\sigma = \ind_{\sigma_1 = i}X_\sigma;\\
\label{eq:op_toy_Fock_last}a_i^\times \Omega = \Omega, \quad &a_i^\times X_\sigma = \ind_{\sigma_1 \neq i}X_\sigma.
\end{align}

\section{Embedding of the toy Fock space into the full Fock space}

Our aim is now to show that the free toy Fock space can be realized as a closed subspace of the full (or free) Fock space $\Fock = \F(L^2(\R_+;\C))$ of square integrable functions. What is more, to each partition of $\R_+$ we shall associate such an embedding, and, as we shall see in the next section, when the diameter of the partition becomes small, one can approximate the full Fock space with the (much simpler) toy Fock space. 

Let $\S = \{0=t_0 < t_1 < \cdots < t_n < \cdots\}$ be a partition of $\R_+$ of diameter $\delta(\S) = \sup_i |t_{i+1} - t_i|$. The main idea of \cite{attal-art} was to decompose the symmetric Fock space of $L^2(\R_+)$ along the partition $\S$. In our free setting we have an analogue exponential property (see eq. (\ref{eq:exp})):
\[\Fock = \invstackrel{\bigstar}{i \in \N} \Fock_i,\]
where $\Fock_i = \F(L^2[t_i, t_{i+1}))$, the countable free product being defined with respect to the vacuum functions. Inside each Fock space $\Fock_i$, we consider two distinguished functions: the vacuum function $\Omega_i$ and the normalized indicator function of the interval $[t_i, t_{i+1})$:
\[X_i = \frac{\ind_{[t_i, t_{i+1})}}{\sqrt{t_{i+1} - t_i}} = \frac{\ind_{t_{i+1}} - \ind_{t_i}}{\sqrt{t_{i+1} - t_i}}.\]
These elements span a 2-dimensional vector space $\C\Omega_i \oplus \C X_i$ inside each $\Fock_i$. The toy Fock space associated to the partition $\S$ is the free product of these two-dimensional vector spaces:
\[\TFock(\S) = \invstackrel{\bigstar}{i \in \N} (\C\Omega_i \oplus \C X_i).\]
$\TFock(\S)$ is a closed subspace of the full Fock space $\Fock$ and it is naturally isomorphic (as a countable free product of two-dimensional spaces) to the abstract free toy Fock space $\TFock$ defined in the previous section. It is spanned by the orthonormal family $\{X_\sigma\}_{\sigma \in \W^*}$, where $X_\sigma = X_\sigma(\S)$ is defined by
\[X_\sigma = X_{\sigma_1} \otimes X_{\sigma_2} \otimes \cdots \otimes X_{\sigma_n} = \left[(x_1, \ldots, x_n) \mapsto  \frac{\prod_{j=1}^n\ind_{[t_{\sigma_j} , t_{\sigma_j+1})}(x_j)}{\prod_{j=1}^n\sqrt{t_{\sigma_j+1} - t_{\sigma_j}}}\right],\]
with $\sigma=[\sigma_1, \ldots, \sigma_n]$. We shall denote by $P_\S \in \B(\Fock)$ the orthogonal projector on $\TFock(\S)$. For a function $f \in \Fock$, which admits a decomposition $f=f_0 \Omega + \sum_{n \geq 1} f_n$ with $f_0 \in \C$ and $f_n \in L^2(\R_+^n)$, the action of $P_\S$ is straightforward to compute:
\[P_\S f = f_0 \Omega + \sum_{n \geq 1} \sum_{\sigma \in \W_n^*} \scalar{X_\sigma}{f_n} X_\sigma,\]
where the scalar products are taken in the corresponding $L^2$ spaces.

% \begin{pr}
% Let $f=f_0 \Omega + \sum_{n \geq 1} f_n$ be a function of $\Fock$ with $f_0 \in \C$ and $f_n \in L^2(\R_+^n)$. Then
% \[P_\S f = f_0 \Omega + \sum_{n \geq 1} \sum_{\sigma \in \W_n^*} \scalar{X_\sigma}{f_n} X_\sigma,\]
% where the scalar product is taken in the corresponding $L^2$ space.
% \end{pr}
% \begin{proof}
% Let $g = \sum_{\sigma \in \W^*} g_\sigma X_\sigma$ be a generic element of $\TFock(\S)$. We have
% \[\scalar{g}{f}_{\F} = \overline{g_\emptyset}f_0 + \sum_{n \geq 1} \sum_{\sigma \in \W^*_n} \overline{g_\sigma} \scalar{X_\sigma}{f_n}_{L^2(\R_+^n)}.\]
% Given the orthogonality of $X_\sigma$ and $X_\tau$ for two distinct words $\sigma, \tau \in \W$, one finds that $\scalar{g}{f} = \scalar{g}{P_\S f}$ for all $g \in \TFock(\S)$ and the conclusion follows. 
% \end{proof}

We ask now how the basic operators $a^\e_t$, $\e \in \{+, -, \circ, \times\}$, $t \in \R^+$ of the free Fock space relate to their discrete counterparts $a^\e_i$. In order to do this, we consider the following rescaled restrictions of $a^+_t$,  $a^-_t$ and $a^\circ_t$ on the toy Fock space $\TFock(\S)$:
\begin{align}
\label{eq:approx_first}a_i^+(\S) &= P_\S\frac{a^+_{t_{i+1}} - a^+_{t_i}}{\sqrt{t_{i+1} - t_i}}P_\S = P_\S a^+\left( \frac{\ind_{[t_i, t_{i+1})}}{\sqrt{t_{i+1} - t_i}} \right)P_\S;\\
a_i^-(\S) &= P_\S\frac{a^-_{t_{i+1}} - a^-_{t_i}}{\sqrt{t_{i+1} - t_i}}P_\S = P_\S a^-\left( \frac{\ind_{[t_i, t_{i+1})}}{\sqrt{t_{i+1} - t_i}} \right)P_\S;\\
\label{eq:approx_last}a_i^\circ(\S) &= P_\S(a^\circ_{t_{i+1}} - a^\circ_{t_i})P_\S  = P_\S a^\circ\left( \ind_{[t_i, t_{i+1})} \right)P_\S.
%a_i^\times(\S) &= P_\S \frac{a^\times_{t_{i+1}} - a^\times_{t_i}}{t_{i+1} - t_i}P_\S  =  P_\S a^\times\left( \frac{\ind_{[t_i, t_{i+1})}}{t_{i+1} - t_i} \right)P_\S = P_\S.
\end{align}
The operators $a_i^\e(\S) \in \B(\Fock)$ are such that $a_i^\e(\S)(\TFock(\S)) \subset \TFock(\S)$ and they vanish on $\TFock(\S)^\bot$, so one can also see them as operators on $\TFock(\S)$. For $\e = \times$, one can not define $a_i^\times(\S)$ from $a_t^\times$ as it was done in eq. (\ref{eq:approx_first}) -- (\ref{eq:approx_last}). Instead, we define it as the linear extension of $a_i^\times$ (via the isomorphism $\TFock \simeq \TFock(\S)$) which vanishes on $\TFock(\S)^\bot$. Hence, $a_i^\times(\S) = P_\S(\Id-a_i^\circ(\S))P_\S$.
\begin{pr}
For $\e \in \{+, -, \circ, \times\}$, the operators $a^\e_i(\S)$, acting on the toy Fock space $\TFock(\S)$,  behave in the same way as their discrete counterparts $a^\e_i$.
\end{pr}
\begin{proof}
For each $\sigma = [\sigma_1, \sigma_2, \ldots, \sigma_n] \in \W^*$, consider the corresponding basis function of $\TFock(\S)$:
\[X_\sigma(\S) = \frac{\ind_\sigma(\S)}{\prod_{j=1}^n\sqrt{t_{\sigma_j+1} - t_{\sigma_j}}}, \] 
where $\ind_\sigma(\S)$ is the indicator function of the rectangle $\times_{j=1}^n[t_{\sigma_j}, t_{\sigma_j+1})$. We have:
\begin{align*}
a_i^+(\S)X_\sigma(\S) &= P_\S\frac{a^+(\ind_{[t_i, t_{i+1})})}{\sqrt{t_{i+1} - t_i}}X_\sigma(\S) = P_\S X_{[i,\sigma]}(\S) = \ind_{\sigma_1 \neq i}X_{[i,\sigma]}(\S), \\
a_i^-(\S)X_\sigma(\S) &= P_\S\frac{a^-(\ind_{[t_i, t_{i+1})})}{\sqrt{t_{i+1} - t_i}}X_\sigma(\S) = P_\S \ind_{\sigma_1 = i}X_{[\sigma_2, \ldots, \sigma_n]}(\S) = \ind_{\sigma_1 = i}X_{[\sigma_2, \ldots, \sigma_n]}(\S), \\
a_i^\circ(\S)X_\sigma(\S) &= P_\S a^\circ(\ind_{[t_i, t_{i+1})})X_\sigma(\S) = P_\S \ind_{\sigma_1 = i}X_\sigma(\S) = \ind_{\sigma_1 = i}X_\sigma(\S).
%a_i^\times(\S)X_\sigma(\S) &= P_\S\frac{a^\times(\ind_{[t_i, t_{i+1})})}{t_{i+1} - t_i}X_\sigma(\S) =??? = \ind_{\sigma_1 \neq i}X_{\sigma}(\S). 
\end{align*}
These relations are identical to the action of the corresponding operators $a^\e_i$ on the abstract toy Fock space $\TFock \simeq \TFock(\S)$ (compare to eq. (\ref{eq:op_toy_Fock_first}) -- (\ref{eq:op_toy_Fock_next_to_last})). For $a^\times_i(\S)$, the conclusion is immediate from the last equation above and its definition:
\[a_i^\times(\S)X_\sigma(\S) = P_\S [\Id-a_i^\circ(\S)]X_\sigma(\S) = X_\sigma(\S) -  \ind_{\sigma_1 = i}X_\sigma(\S) = \ind_{\sigma_1 \neq i}X_\sigma(\S).\]
\end{proof}

\section{Approximation results}

This section contains the main result of this work, Theorem \ref{thm:approx}. We show that the toy Fock space $\TFock(\S)$ together with its operators $a_i^\e$ approach the full Fock space $\Fock$ and its operators $a_t^\e$ when the diameter of the partition $\S$ approaches 0. 

Let us consider a sequence of partitions $\S_n = \{0 = t_0^{(n)} < t_1^{(n)} < \cdots < t_k^{(n)} < \cdots\}$ such that $\delta(\S_n) \to 0$. In order to lighten the notation, we put $\TFock(n) = \TFock(\S_n)$, $P_n = P_{\S_n}$ and $a^\e_i(n) = a^\e_i(\S_n)$.

\begin{theo}\label{thm:approx}For a sequence of partitions $\S_n$ of $\R_+$ such that $\delta(\S_n) \to 0$, one has the following approximation results:
\begin{enumerate}
\item For every $f \in \Fock$, $P_n f \to f$.
\item For all $t \in \R_+$, the operators 
\begin{align*}
a^\pm_t(n) &= \sum_{i:t_i^{(n)} \leq t} \sqrt{t_{i+1}^{(n)} - t_i^{(n)}} a_i^\pm(n),\\
a^\circ_t(n) &= \sum_{i:t_i^{(n)} \leq t}a_i^\circ(n),\\
a^\times_t(n) &= \sum_{i:t_i^{(n)} \leq t} \left(t_{i+1}^{(n)} - t_i^{(n)}\right ) a_i^\times(n)
\end{align*}
converge strongly, when $n \to \iy$, to $a_t^\pm$, $a_t^\circ$ and $a_t^\times$ respectively.
\end{enumerate}
\end{theo} 
\begin{proof}
For the fist part, consider a (not necessarily adapted) word $\sigma = [\sigma_1, \ldots, \sigma_k]$ and denote by $\ind_\sigma^{(n)}$ the indicator function of the rectangle $\times_{j=1}^k[t^{(n)}_{\sigma_j}, t^{(n)}_{\sigma_j+1})$ of $\R_+^k$. It is a classical result in integration theory that the simple functions $\{\ind_\sigma^{(n)}\}_{\sigma \in \W_k, n \geq 1}$ are dense in $L^2(\R_+^k)$ for all $k$. It is obvious that the result still holds when replacing $\W_k$ with the set of adapted words $\W_k^*$.

As for the second statement of the theorem, let us start by treating the case of $a_t^+$. For fixed $n$ and $t$, let $t^{(n)} = t_{i+1}^{(n)}$ , where $i$ is the last index appearing in the definition of $a^+_t(n)$, i.e. $t_i^{(n)} \leq t < t_{i+1}^{(n)}$. With this notation, we have $a^+_t(n) = \sum_{i:t_i^{(n)} \leq t} \sqrt{t_{i+1}^{(n)} - t_i^{(n)}} a_i^+(n) = P_n a_{t^{(n)}}^+P_n$. Hence, for any function $f \in \F$, we obtain:
\begin{align*}
\norm{a^+_t(n)f - a_t^+ f } &= \norm{ P_n a_{t^{(n)}}^+P_nf -  a_t^+ f} \leq\\
&\leq \norm{ P_n a_{t^{(n)}}^+P_nf -  P_n a_{t^{(n)}}^+f} +  \norm{ P_n a_{t^{(n)}}^+f - P_n a_t^+f } + \norm{ P_n a_t^+f - a_t^+f }\leq\\
&\leq \norm{ P_n a_{t^{(n)}}^+}\norm{(P_n-I)f} +  \norm{ P_n a^+\ind_{[t,t^{(n)})} }\norm{f} + \norm{(P_n -I)(a_t^+f)}.
\end{align*}
By the first point, $P_n \to I$ strongly, hence the first and the third terms above converge to 0. The norm of the operator appearing in the second term is bounded by the $L^2$ norm of $\ind_{[t,t^{(n)})}$ which is infinitely small when $n \to \iy$. Hence, the entire quantity converges to 0 and we obtained the announced strong convergence. The proof adapts easily to the cases of $a_t^-$ and $a_t^\circ$.

Finally, recall that  $a_i^\times(n) = P_n(\Id-a_i^\circ(n))P_n$. Hence, with the same notation as above, 
\[\sum_{i:t_i^{(n)} \leq t} \left(t_{i+1}^{(n)} - t_i^{(n)}\right) a_i^\times(n) = t^{(n)}P_n + \sum_{i:t_i^{(n)} \leq t} \left(t_{i+1}^{(n)} - t_i^{(n)}\right) a_i^\circ(n).\]
The second term above converges to zero in the strong operator topology thanks to the factor $t_{i+1}^{(n)} - t_i^{(n)}$ which is less than $\delta(\S_n)$, and thus we are left only with $ t^{(n)}P_n$ which converges, by the first point, to $t \cdot \Id$.
\end{proof}

\section{Applications to free probability theory}\label{sec:free_prob}

This section is more probabilistic in nature. We use the previous approximation result to show that the free Brownian motion and the free Poisson operators can be approached, in the strong operator topology, by sums of free Bernoulli-distributed operators living on the free toy Fock space. We obtain, as corollaries, already known free Donsker-like convergence results.

Let us start by recalling some basic facts about free noises and their realization on the free Fock space $\Fock$. The free Brownian motion $W_t$ and the free Poisson process $N_t$ were constructed in \cite{speicher-90} as free analogues of the classical Brownian motion (or Wiener process) and, respectively, classical Poisson jump processes. Recall that a process with stationary and freely independent increments is a collection of non commutative self-adjoint random variables $(X_t)_t$  with the following properties:
\begin{enumerate}
\item For all $s<t$, $X_t-X_s$ is free from the algebra generated by $\{X_u, u \leq s\}$;
\item The distribution of $X_t-X_s$ depends only on $t-s$.
\end{enumerate}
A free Brownian motion is a process with stationary and freely independent increments $(W_t)_t$ such that the distribution of  $W_t-W_s$ is a \emph{semi-circular} random variable of mean 0 and variance $t-s$. Recall that a standard (i.e. mean zero and variance one) semicircular random variable has distribution
\[d\mu(x) = \frac{1}{2\pi}\sqrt{4-x^2}\ind_{[-2,2]}(x) dx.\]
If $X$ is a standard semicircular random variable, then $(t-s)X$ is semicircular of variance $(t-s)$. In an analogue manner, a free Poisson process is a process with stationary and freely independent increments $(N_t)_t$ such that the distribution of  $N_t-N_s$ is a \emph{free Poisson} random variable of parameter $\lambda=t-s$. In general, the density of a free Poisson random variable is given by
\[d\nu_\lambda(x) = 
\begin{cases}
\frac{\sqrt{4\lambda-(x-1-\lambda)^2}}{2\pi x}\chi(x)dx & \text{if } \lambda \geq 1,\\
(1-\lambda) \delta_0 + \frac{\sqrt{4\lambda-(x-1-\lambda)^2}}{2\pi x}\chi(x)dx & \text{if } 0<\lambda < 1,\\
\end{cases}
\]
where $\chi$ is the indicator function of the interval $[(1-\sqrt{\lambda})^2,(1+\sqrt{\lambda})^2]$.

The free Brownian motion and the free Poisson process can be realized on the full Fock space $\Phi$ as $W_t = a^+_t + a^-_t$ and, respectively, $N_t = a^+_t + a^-_t + a^\circ_t + t\cdot \Id$. Generalization of these processes and stochastic calculus were considered in \cite{bia-spe, boz-spe, gss}.

For the sake of simplicity, throughout this section we shall consider the sequence of partitions $\S_n = \{k/n; k \in \N\}$; obviously $\delta(\S_n) = \frac 1 n \to 0$. The following result is an easy consequence of Theorem \ref{thm:approx}.
 
\begin{pr}\label{pr:brownian}
On $\TFock(n)$, consider the operator $X^{(n)}_i = a^+_i + a^-_i$, $i \in \N$. Then
\begin{enumerate}
\item For all $n \geq 1$, the family $\{X^{(n)}_i\}_{i \in \N}$ is a free family of Bernoulli random variables of distribution $\frac 1 2 \delta_{-1} + \frac 1 2 \delta_1$.
\item For all $t \in \R_+$, the operator
\[W_t^{(n)} = \frac{1}{\sqrt n}\sum_{i=0}^{\lfloor nt \rfloor}X_i^{(n)}\]
converges in the strong operator topology, when $n \to \iy$, to the operator of free Brownian motion $W_t =  a^+_t + a^-_t$.
\end{enumerate}
\end{pr}

Let us show now that the strong operator convergence implies the convergence in distribution of the corresponding processes. Let $t_1, \ldots, t_s \in \R_+$ and $k_1, \ldots, k_s \in \N$. Since, by the previous result, $W_t^{(n)} \to W_t$ strongly, and multiplication is jointly strongly continuous on bounded subsets, we get that $(W_{t_1}^{(n)})^{k_1} \cdots (W_{t_s}^{(n)})^{k_s} \to W_{t_1}^{k_1}\cdots W_{t_s}^{k_s}$ strongly. Strong convergence implies convergence of the inner products $\scalar{\Omega}{\cdot \Omega}$ and thus the following corollary (which is a direct consequence of the Free Central Limit Theorem \cite{nica-speicher, vdn}) holds. 
\begin{cor}
The distribution of the family $\{W^{(n)}_t\}_{t \in \R_+}$ converges, as $n$ goes to infinity, to the distribution of a free Brownian motion $\{W_t\}_{t \in \R_+}$.
\end{cor}

We move on to the free Poisson process $N_t$ and we state the analogue of Proposition \ref{pr:brownian}.
%. In fact, we shall consider the free \emph{compensated} Poisson process, $M_t = N_t - t \cdot \Id = a^+_t + a^-_t + a^\circ_t$, as an operator on the free Fock space $\Fock$. We state now the analogue of Proposition \ref{pr:brownian}.

\begin{pr}
On $\TFock(n)$, consider the operator $Y^{(n)}_i = a^+_i + a^-_i + \sqrt{n} a^\circ_i + \frac{1}{\sqrt{n}}a^\times_i$. Then
\begin{enumerate}
\item For all $n \geq 1$, the family $\{Y^{(n)}_i\}_{i \in \N}$ is a free family of Bernoulli random variables of distribution $\frac{1}{n+1} \delta_{\frac{n+1}{\sqrt{n}}} + \frac{n}{n+1} \delta_{0}$.
\item For all $t \in \R_+$, the operator
\[N_t^{(n)} = \frac{1}{\sqrt n}\sum_{i=0}^{\lfloor nt \rfloor}Y_i^{(n)}\]
converges strongly, when $n \to \iy$, to the operator of the free Poisson process $N_t =  a^+_t + a^-_t + a^\circ_t + a^\times_t$.
\end{enumerate}
\end{pr}
\begin{proof}
As an operator on $\C^2$, $Y^{(n)}_i$ has the form
\[Y^{(n)}_i=\begin{bmatrix}
\frac{1}{\sqrt{n}}&1\\
1&\sqrt{n}
\end{bmatrix}.\]
The $k$-th moment of $Y^{(n)}_i$ is easily seen to be given by the formula
\[\scalar{\Omega}{(Y^{(n)}_i)^k \Omega} = \frac{1}{n+1}\left( \frac{n+1}{\sqrt{n}}\right)^k,\]
which is the same as the $k$-th moment of the probability distribution $\frac{1}{n+1} \delta_{\frac{n+1}{\sqrt{n}}} + \frac{n}{n+1} \delta_{0}$, and the first part follows. 
For the second part, we have
\begin{align*}
N_t^{(n)} &= \frac{1}{\sqrt n}\sum_{i=0}^{\lfloor nt \rfloor}Y_i^{(n)} = \sum_{i; t_i^{(n)}\leq t}\left[\frac{1}{\sqrt n}a_i^+ + \frac{1}{\sqrt n}a_i^- + a_i^\circ + \frac{1}{n}a_i^\times \right]=\\
 &=\sum_{i; t_i^{(n)}\leq t} \sqrt{t_{i+1}^{(n)} - t_i^{(n)}} \left(a_i^+(n) + a_i^-(n)\right) + \sum_{i; t_i^{(n)}\leq t}a_i^\circ + \sum_{i; t_i^{(n)}\leq t} \left(t_{i+1}^{(n)} - t_i^{(n)}\right) a_i^\times.
\end{align*}
Using Theorem \ref{thm:approx}, one obtains $N_t^{(n)} \to N_t$ in the strong operator topology.
\end{proof}

Again, we obtain as a corollary the convergence in distribution of the process $(N_t^{(n)})_t$ to the free Poisson process, which is in fact a reformulation of the Free Poisson limit theorem (\cite{nica-speicher}, pp. 203).

\begin{cor}
The distribution of the family $\{N^{(n)}_t\}_{t \in \R_+}$ converges, as $n$ goes to infinity, to the distribution of a free Poisson process $\{N_t\}_{t \in \R_+}$.
\end{cor}

\section{Higher multiplicities}\label{sec:multi}

We generalize now the previous construction of the free toy Fock space by replacing $\C^2$ with the $N+1$-dimensional complex Hilbert space $\C^{N+1}$. Much of what was done in the $\C^2$ extends easily to the generalized case, so we only sketch the construction, leaving the details to the reader (for an analogue setup in the symmetric Fock space, see \cite{attal-pautrat}). In what follows, $N \geq 1$ is a fixed integer, called the \emph{multiplicity} of the Fock space. 

Start with a countable family of copies of $\C^{N+1}$, each endowed with a fixed basis $(\Omega, X^1, \ldots , X^N)$. We shall sometimes note $X^0 = \Omega$. We introduce the free toy Fock space of multiplicity $N$ (see Section \ref{sec:free_toy_fock}):
\[\TFock = \invstackrel{\bigstar}{i \in \N} \C^{N+1}(i),\]
where the countable tensor product is defined with respect to the stabilizing sequence of vectors $\Omega(i) \in \C^{N+1}(i)$. An orthonormal basis of this space is indexed by the set $\W^{N*}$ of generalized adapted words $\sigma = [(i_1, j_1) , (i_2, j_2), \ldots, (i_n, j_n)]$, where $n \in \N$, $i_1 \neq i_2 \neq \cdots \neq i_n$ and $j_1, \ldots, j_n \in \{1, \ldots, N\}$, the corresponding basis element being $X_\sigma = X^{j_1}(i_1) \otimes X^{j_2}(i_2) \otimes \cdots \otimes X^{j_n}(i_n)$.

On each copy of $\C^{N+1}$ we introduce the matrix units $a^i_j$ defined by
\[a^i_jX^k = \delta_{ik}X^j, \quad i,j,k=0,1,\ldots, N.\]

We shall now show how the discrete structure of the free toy Fock space of multiplicity $N$ approximates the free Fock space $\Fock = \F(L^2(\R_+; \C^N))$. To this end, consider a partition $\S = \{0=t_0 < t_1 < \cdots < t_n < \cdots\}$  of $\R_+$ and recall the decomposition of the free Fock space of multiplicity $N$ as a free product of ``smaller'' Fock spaces:
\[ \F(L^2(\R_+; \C^N)) = \invstackrel{\bigstar}{i \in \N} \F(L^2([t_i, t_{i+1}); \C^N)).\]
In each factor of the free product we consider $N+1$ distinguished functions: the constant function $\Omega_i$ (sometimes denoted by $X^0(i)$) and the normalized indicator functions
\[X^j(i) = \frac{\ind^j_{[t_i, t_{i+1})}}{\sqrt{t_{i+1} - t_i}} = \frac{\ind^j_{t_{i+1}} - \ind^j_{t_i}}{\sqrt{t_{i+1} - t_i}}, \quad 1\leq j \leq N,\]
where $\ind^j_A(x)=(0, \ldots, 0, 1, 0, \dots, 0)^\top$ with the 1 in the $j$-th position if $x \in A$ and 0 otherwise. For a generalized word $\sigma = [(i_1, j_1) , (i_2, j_2), \ldots, (i_n, j_n)]$, define the element $X_\sigma(\S) \in \Fock$ by
\[X_\sigma(\S) = X^{j_1}(i_1) \otimes \cdots \otimes X^{j_n}(i_n)=[(x_1, \ldots, x_n) \mapsto  \frac{\prod_{k=1}^n\ind^{j_k}_{[t_{i_k} , t_{i_k+1})}(x_k)}{\prod_{k=1}^n\sqrt{t_{i_k+1} - t_{i_k}}}],\]
with $\sigma = [(i_1, j_1) , (i_2, j_2), \ldots, (i_n, j_n)]$. The toy Fock space associated to $\S$ (denoted by $\TFock(\S)$) is the span of $X_\sigma(\S)$ for all generalized adapted words $\sigma \in \W^{N*}$. $\TFock(\S)$ is a closed subspace of the full Fock space $\Fock$ and it is naturally isomorphic to the abstract toy Fock space of multiplicity $N$, $\TFock$. For a given sequence of refining partitions $\S_n$ whose diameters converge to zero, the toy Fock spaces and the operators $a^i_j$ approximate the Fock space $\Fock$ and its corresponding operators (compare with Theorem \ref{thm:approx}):
\begin{theo}\label{thm:approx_multi}Let $\Fock$ be the free Fock space of multiplicity $N$ and $\S_n$ a sequence of refining partitions of $\R_+$ such that $\delta(\S_n) \to 0$. Then one has the following approximation results:
\begin{enumerate}
\item For every $f \in \Fock$, $P_n f \to f$.
\item For $i,j \in \{0,1, \ldots, N\}$, define $\e_{ij} = \frac 1 2 (\delta_{0i} + \delta_{0j})$. Then, for all $t \in \R_+$, the operators 
\[\sum_{k:t_k^{(n)} \leq t} (t_{k+1}^{(n)} - t_k^{(n)})^{\e_{ij}} a^i_j(k)\]
converge strongly, when $n \to \iy$, to $a^i_j(t)$.
\end{enumerate}
\end{theo} 

\subsection*{An example for $N=2$}
Let us end this section by constructing an approximation of a two-dimensional free Brownian motion constructed on a free Fock space of multiplicity $N=2$. To this end, define the free Fock space $\Phi = \F(L^2(\R_+;\C^2))$ and its discrete approximation, the free toy Fock space $\TFock = \bigstar_{k \in \N} \C^3_{(k)}$. The simplest realization of two freely independent free Brownian motions on $\Fock$ is the pair of operator processes $W_1(\cdot), W_2(\cdot) \in \B(\Fock)$ defined by:
\[W_1(t) = a^0_1(t) + a^1_0(t) \text{ and } W_2(t) = a^0_2(t) + a^2_0(t).\]
First of all, it is obvious that both $W_1(\cdot)$ and $W_2(\cdot)$ are free Brownian motions (see Section \ref{sec:free_prob}). Moreover, the families $(W_1(t))_t$ and $(W_2(t))_t$ are freely independent since the functions $\ind^1_s$ and $\ind^2_t$ are orthogonal in $ \F(L^2(\R_+;\C^2))$ (see Proposition \ref{pr:free_orth}).
We consider, as we did in Section \ref{sec:free_prob}, the sequence of refining partitions $\S_n = \{k/n; k \in \N\}$. We introduce the following two families of operators:
\begin{align*}
Y_1(k) &= a^0_1(k) + a^1_0(k), \\
Y_2(k) &= a^0_2(k) + a^2_0(k),
\end{align*}
and respectively
\begin{align*}
Z_1(k) &= a^0_1(k) + a^1_0(k) - a^2_2(k),\\
Z_2(k) &= a^0_2(k) + a^2_0(k)-[a^1_2(k) + a^2_1(k) + a^2_2(k)],
\end{align*}
for $k \in \N$. It follows from Theorem \ref{thm:approx_multi} that for all $t \in \R_+$, both families are approximations of a two-dimensional Brownian motion:
\[\frac{1}{\sqrt n} \left(\sum_{i=0}^{\lfloor nt \rfloor}Y_1(n), \sum_{i=0}^{\lfloor nt \rfloor}Y_2(n) \right) \invstackrel{\longrightarrow}{n \to \iy} \left(W_1(t), W_2(t)\right)\]
and
\[\frac{1}{\sqrt n} \left(\sum_{i=0}^{\lfloor nt \rfloor}Z_1(n), \sum_{i=0}^{\lfloor nt \rfloor}Z_2(n) \right) \invstackrel{\longrightarrow}{n \to \iy} \left(W_1(t), W_2(t)\right),\]
where the limits hold in the strong operator topology.
%Let us justify this fact, at least for the $(Z_1, Z_2)$ couple. 
However, the building blocks of these approximating processes have completely different behaviors at fixed $k$. To start, note that the self-adjoint operators $Y_1(k)$ and $Y_2(k)$, represented, in the basis $(\Omega, X^1,  X^2)$, by the hermitian matrices 
\[ Y_1 = \begin{bmatrix}
0&1&0\\
1&0&0\\
0&0&0
\end{bmatrix}
\text{ and } Y_2 = \begin{bmatrix}
0&0&1\\
0&0&0\\
1&0&0
\end{bmatrix}\]
do not commute. Hence, they do not admit a classical joint distribution, i.e. it does not exist a probability measure $\mu$ on $\R^2$ such that
\begin{equation}\label{eq:joint_distrib}
\int_{\R^2}y_1^my_2^n d\mu(y_1,y_2) = \scalar{\Omega}{Y_1^mY_2^n\Omega}.
\end{equation}
On the contrary, for each $k$, the operators $Z_1(k)$ and $Z_2(k)$, which act on $\C^3$ as the matrices
\[ Z_1 = \begin{bmatrix}
0&1&0\\
1&0&0\\
0&0&-1
\end{bmatrix}
\text{ and } Z_2 = \begin{bmatrix}
0&0&1\\
0&0&-1\\
1&-1&-1
\end{bmatrix},\]
 commute and they admit the following classical joint distribution (in the sense of equation (\ref{eq:joint_distrib})):
\[\mu = \frac{1}{2}\delta_{(1,0)} + \frac{1}{3}\delta_{(-1,1)} + \frac{1}{6}\delta_{(-1,-2)}.\]
More details on high multiplicity Fock spaces and the analogue construction in the commutative case can be found in \cite{attal-emry, attal-pautrat}.


\begin{thebibliography}{10}

\bibitem{attal-art}
Attal, S. {\it Approximating the Fock space with the toy Fock space}. S\'eminaire de Probabilit\'es, XXXVI,  477--491, Lecture Notes in Math., 1801, Springer, Berlin, 2003.

\bibitem{attal-emry}
Attal, S., \'Emery, M. {\it \'Equations de structure pour des martingales vectorielles}, S\'eminaire de Probabilit\'es, XXVIII, 256--278, Lecture Notes in Math., 1583, Springer, Berlin, 1994. 

\bibitem{attal-pautrat}
Attal, S., Pautrat, Y. {\it From $(n+1)$-level atom chains to $n$-dimensional noises}, Ann. Inst. H. Poincar\'e Probab. Statist. 41 (2005), no. 3, 391--407.

\bibitem{attal-pautrat2}
Attal, S., Pautrat, Y. {\it From repeated to continuous quantum interactions}, Ann. Henri Poincar\'e 7 (2006), no. 1, 59--104. 

\bibitem{bia-spe}
Biane, P., Speicher, R. {\it Stochastic calculus with respect to free Brownian motion and analysis on Wigner space}, Probab. Theory Related Fields 112 (1998), no. 3, 373--409. 

\bibitem{boz-spe}
Bo\.zejko, M., Speicher, R. {\it An example of a generalized Brownian motion},  Comm. Math. Phys.  137  (1991),  no. 3, 519--531.

\bibitem{BJM}
Bruneau, L., Joye, A., Merkli, M. {\it Asymptotics of repeated interaction quantum systems}, J. Funct. Anal. 239 (2006), no. 1, 310--344. 

\bibitem{BP}	
Bruneau, L., Pillet, C.-A. {\it Thermal relaxation of a QED cavity}, preprint available at \url{http://hal.archives-ouvertes.fr/hal-00325206/en/}.

\bibitem{gss}
Glockner, P., Sch\"urmann, M., Speicher, R. {\it Realization of free white noises}, Arch. Math. (Basel) 58 (1992), no. 4, 407--416.

\bibitem{hiai-petz}
Hiai, F., Petz, D. {\it The semicircle law, free random variables and entropy} Mathematical Surveys and Monographs, 77. American Mathematical Society, Providence, RI, 2000.

\bibitem{nica-speicher}
Nica, A., Speicher, R. {\it Lectures on the combinatorics of free probability}, Cambridge University Press, 2006.

\bibitem{speicher-90}
Speicher, R. {\it A new example of ``independence'' and ``white noise''},  Probab. Theory Related Fields  84  (1990),  no. 2, 141--159. 

\bibitem{vdn} Voiculescu, D.V., Dykema, K., Nica, A. \emph{Free random variables}, CRM Monograghs Series No.1, Amer. Math. Soc., Providence, RI, 1992.

\bibitem{voic-LN}
Voiculescu, D.V.  {\it Lecture notes on free probability}, Lecture Notes in Math., 1738, Springer, Berlin, 2000, pp 279--349.
\end{thebibliography}
\end{document}